\documentclass[10pt]{article} 

\usepackage{amsmath,amsthm,amsfonts,amssymb,amscd,latexsym,eucal}
\usepackage[a4paper,top=3cm,bottom=4.5cm,left=2.5cm,right=2.5cm]{geometry}

\newtheorem{corollary}{Corollary}

\newtheorem{lemma}{Lemma}
\newtheorem{proposition}{Proposition}
\newtheorem{remark}{Remark}

\author{Carlo Pandiscia}
\title{An Ergodic Dilation of Completely Positive Maps}
\date{}

\begin{document}
\maketitle

\begin {abstract}
We shall prove the following Stinespring-type theorem: there exists a triple
$(\pi,\mathcal{H},\mathbf{V})$ associated with an unital completely positive map 
$\Phi:\mathfrak{A}\rightarrow \mathfrak{A}$ on
C*-algebra $\mathfrak{A}$ with unit, where $\mathcal{H}$ is a Hilbert space, 
$\pi:\mathfrak{A\rightarrow B}(\mathcal{H})$  
is a faithful representation and 
$\mathbf{V}$ 
is a linear isometry on 
$\mathcal{H}$ 
such that
$\pi(\Phi(a)=\mathbf{V}^*\pi(a)\mathbf{V}$ 
for all $a$ belong to $\mathfrak{A}$. The Nagy dilation theorem, applied to isometry 
$\mathbf{V}$,
allows to construct a dilation of ucp-map,
$\Phi$, in the sense of Arveson, 
that satisfies ergodic properties of a $\Phi $-invariante state $\varphi$ on 
$\mathfrak{A}$, if $\Phi$ admit a $\varphi $-adjoint.
\end {abstract}

\section{Introduction}
A discrete quantum process is a pair $(\mathfrak{M},\Phi)$ consisting of a von
Neumann algebra $\mathfrak{M}$ and a normal unital completely positive map
$\Phi$ on $\mathfrak{M}$. In this work we shall prove that any quantum process
is possible dilate to quantum process where the dynamic $\Phi$ is a
*-endomorphism of a larger von Neumann algebra.
\newline 
In dynamical systems,
the process of dilation has taken different meanings. Here we adopt the
following definition (See Ref. Muhly-Solel \cite{Muh}):
\newline 
Suppose $\mathfrak{M}$ acts on Hilbert space $\mathcal{H}$, a dilation of a quantum process $(\mathfrak{M},\Phi)$ is a quadruple $(\mathfrak{R},\Theta,\mathcal{K},z)$ where $(\mathfrak{R},\Theta)$ is a quantum process with $\mathfrak{R}$ acts on Hilbert space $\mathcal{K}$ and $\Theta$ is a homomorphism (i.e. *-endomorphism on von Neumann algebra $\mathfrak{R}$) with $z:\mathcal{H}\rightarrow\mathcal{K}$ isometric embedding such that:

\begin{itemize}
\item  $z\mathfrak{M}z^*\subset$ $\mathfrak{R}$ \ and $\ z^*\mathfrak{R}z\subset\mathfrak{M};$

\item  $\Phi^{n}(a)=z^*\Theta^{n}(zaz^*)z$ \ \ \ for all $a\in\mathfrak{M}$ and $n\in\mathbb{N};$

\item  $z^*\Theta^{n}(X)z=\Phi^{n}(z^*Xz) $ \ for all $X\in\mathfrak{R}$ and $n\in\mathbb{N}.$
\end{itemize}
Many authors in the past have been applied to problems very similar to the one we described above. We remember the work of Arveson \cite{arv} on the Eo-semigroups, of Baht-Parthasarathy on the dilations of nonconservative dynamical semigroups \cite{bath} and finally, the most recent work of Mhulay-Solel \cite{Muh}.
\newline
We shall prove the existence of dilation using the Nagy theorem for linear contraction (See Fojas-Nagy Ref.\cite{NF}) and of a particular covariat representation obtained through the Stinespring's theorem for completely positive maps (See Stinespring Ref.\cite{Stine}).
\newline
We recall that a covariant representation of discrete quantum process $(\mathfrak{M},\Phi)$ is a triple $(\pi,\mathcal{H},\mathbf{V})$ where $\pi:\mathfrak{M\rightarrow B}(\mathcal{H})$ is a normal faithful representation on the Hilbert space $\mathcal{H}$ and $\mathbf{V}$ is an isometry on $\mathcal{H}$ such that for $a\in \mathfrak{M}$ and $a\in\mathbb{N}$,
\[
\pi(\Phi^n(a))=\mathbf{V}^{n*}\pi(a)\mathbf{V}^n.
\]
Since the covariant representation is faithful and normal, we identify the von neuman algebra $\mathfrak{M}$ with $\pi(\mathfrak{M})$ and in sec. 3 we construct a dilation of the quantum process $(\pi(\mathfrak{M}),\Psi)$ where $\Psi$ is  the following completely positive map $\Psi(\pi(x))=\pi(\Phi(x))$ for all $n\in\mathfrak{M}$.
\newline
In fact, if the triple $(\widehat{\mathbf{V}},\widehat{\mathcal{H}},z)$ is the minimal unitary dilation of isometry $\mathbf{V}$, we can construct a von Neumann algebras  $\widehat{\mathfrak{M}}\subset\mathfrak{B}(\widehat{\mathcal{H}})$ with following properties: $\widehat{\mathbf{V}}^*\widehat{\mathfrak{M}}\widehat{\mathbf{V}}\subset\widehat{\mathfrak{M}}$
and $z^*\widehat{\mathfrak{M}}z=\mathfrak{M}$.
\newline
Of fundamental importance to quantum process theory, is the $\varphi$-adjointness properties. The dynamic $\Phi$ admit a $\varphi$-adjoint (See Kummerer Ref.\cite{kum}) relative to the normal $\Phi$-invariant state $\varphi$ on $\mathfrak{M}$, if there is a normal unital completely positive map $\Phi_{\natural}:\mathfrak{M\rightarrow M}$ 
such that for $a,b\in\mathfrak{M}$,
\[
\varphi(\Phi(a)b)=\varphi(a\Phi_{\natural}(b)).
\]
The relationship between reversible process, modular operator and $\varphi$-adjointness has been studied by Accardi-Cecchini in \cite{accardi} and Majewski in \cite{Maj}.
\newline
In sec. 4 we shall prove that our dilation satisfies ergodic properties of a $\Phi $-invariante state $\varphi$ on $\mathfrak{M}$ if the dynamic $\Phi$ admit a $\varphi$-adjoint.
\newline
More precisely, let $(\mathfrak{R},\Theta)$ be our dilation of quantum process $(\mathfrak{M},\Phi)$, we shall prove that if  
\[
\underset{n\rightarrow\infty}{\lim}\dfrac{1}{n+1}
\sum \limits_{k=0}^{n}|\varphi(a\Phi^{k}(b))-\varphi(a)\varphi(b)|=0, 
\]
for all $a,b\in\mathfrak{M}$, we have 
\[
\underset{n\rightarrow\infty}{\lim}\dfrac{1}{n+1}
\sum \limits_{k=0}^{n}
|\varphi(z^*X\Theta^k(Y)z)-\varphi(z^*Xz)
\varphi(z^*Yz)|=0,
\]
for all $X,Y\in\mathfrak{R}$.
\newline
For generality, we will work with concrete unital C*-algebras $\mathfrak{A}$ and unital completely positive map $\Phi$ (briefly ucp-map). The results obtained are easily extended to the quantum process $(\mathfrak{M},\Phi)$.  
\newline
Before introducing the proof about existence of dilation of discrete quantum process, it is necessary to recall the fundamental Nagy dilation theorem, subject of the next section.

\section{Nagy dilation theorem}
If $\mathbf{V}$ is a linear isometry on Hilbert space $\mathcal{H}$, 
there is a triple $(\widehat{\mathbf{V}},\widehat{\mathcal{H}},\mathbf{Z})$ 
where 
$\widehat{\mathcal{H}}$ 
is a Hilbert space, 
$\mathbf{Z}:\mathcal{H}\mathbf{\rightarrow }\widehat{\mathcal{H}}$
is a lineary isometry, while 
$\widehat{\mathbf{V}}$ 
is an unitary operator on 
$\widehat{\mathcal{H}}$ 
such that for $n\in \mathbb{N}$,
\begin{equation}
\widehat{\mathbf{V}}^n\mathbf{Z}=\mathbf{Z}\mathbf{V}^n, \label{dilatcontraz}
\end{equation}
with the following minimal properties:
\begin{equation}
\widehat{\mathcal{H}}=\bigvee_{k\in \mathbb{Z}}\widehat{\mathbf{V}}
^k\mathbf{Z}\mathcal{H}.\label{dilatcontrazminimal}
\end{equation}
For our purposes it is useful to recall here the structure of the unitary
minimal dilation of a contraction (See Fojas-Nagy Ref.\cite{NF}).
\newline
Let $\mathcal{K}$ be a Hilbert space, by 
$l^{2}(\mathcal{K})$
we denote the Hilbert space 
$\{\xi:\mathbb{N}\rightarrow \mathcal{K}:\sum\limits_{n\geq 0}\left\Vert \xi(n) \right\Vert ^2<\infty\}.$
\newline
We now get the orthogonal projection 
$\mathbf{F}=\mathbf{I}-\mathbf{VV}^*$ 
and the following Hilbert space 
$\widehat{\mathcal{H}}=\mathcal{H}\oplus l^{2}(\mathbf{F}\mathcal{H})$ 
and define the following unitary operator on the Hilbert space $\widehat{\mathcal{H}}$:
\begin{equation*}
\widehat{\mathbf{V}}=\left\vert 
\begin{array}{cc}
\mathbf{V} & \mathbf{F}\Pi _{0} \\ 
\mathbf{0} & \mathbf{W}
\end{array}
\right\vert,
\end{equation*}
where for each $j\in \mathbb{N}$ 
we have set with 
$\Pi_{j}:l^{2}(\mathbf{F}\mathcal{H})\rightarrow \mathcal{H}$ 
the canonical projections:
\begin{equation*}
\Pi_{j}(\xi _{0},\xi _{1}...\xi _{n}...) =\xi_{j},
\end{equation*}
while 
$\mathbf{W}:l^{2}(\mathbf{F}\mathcal{H})\rightarrow l^{2}(\mathbf{F}\mathcal{H})$ 
is the linear operator
\begin{equation*}
\mathbf{W}( \xi _{0},\xi _{1}...\xi _{n}...) =(\xi_{1},\xi
_{2}...),
\end{equation*}
for all 
$(\xi _{0},\xi _{1}...\xi _{n}...)\in l^{2}(\mathbf{F}\mathcal{H})$.
\newline 
If 
$Z:\mathcal{H}\rightarrow\widehat{\mathcal{H}}$ 
is the isometry defined by 
$\mathbf{Z}\Psi=\Psi\oplus0$ 
for all
$\Psi\in\mathcal{H}$, 
it's simple to prove that the relationships \ref{dilatcontraz} and \ref{dilatcontrazminimal} are given.
\newline 
We observe that for each 
$n\in\mathbb{N}$ 
we have
\begin{equation}
\widehat{\mathbf{V}}^n=\left\vert
\begin{array}
[c]{cc}
\mathbf{V}^n & C(n)  \\
\mathbf{0} & \mathbf{W}^n
\end{array}
\right\vert,
\end{equation}
where 
$C(n):l^{2}(\mathbf{F}\mathcal{H})\rightarrow\mathcal{H}$ 
are the following operators:
\begin{equation}
C(n)={\textstyle\sum\limits_{j=1}^{n}}
\mathbf{V}^{n-j}\mathbf{F}\Pi_{j-1},\ \ \ \ n\geq1.
\end{equation}
Furthermore, for each 
$n,m\in\mathbb{N}$ 
we obtain:
\begin{equation} 
\Pi_{n}\mathbf{W}^{m}=\Pi_{n+m}\text{\ \ \ and\ \ \ }\Pi_{n}\mathbf{W}^{m^*}=\left\{
\begin{array}
[c]{cc}
\Pi_{n-m} & n\geq m\\
0 & n<m
\end{array}
\right. , \label{rel-1}
\end{equation}
since
\[
\mathbf{W}^{m\ast}(\xi_{0},\xi_{1}...\xi_{n}..)=(0,0....0,\overset{m+1}{\overbrace{\xi_{0}}},\xi_{1}...)  ,
\]
while for each $k$ and $p$ natural number, we obtain:
\begin{equation}
\Pi _{p}\mathbf{C}(k) ^*=\left\{ 
\begin{array}{cc}
\mathbf{FV}^{( k-p-1) ^*} & k>p \\ 
\mathbf{0} & \text{elsewhere}
\end{array}
\right.  \label{rel-2}
\end{equation}
since
\[
C(k)^*\Psi=(\overset{k-time}{\overbrace
{\mathbf{FV}^{(k-1)^*}\Psi......\mathbf{FV}^{^*}
\Psi,\mathbf{F}\Psi}},0,.0..)  .
\]
for all 
$\Psi\in\mathcal{H}$.

\section{Invariant algebra}

Let be $\mathfrak{A\subset B}(\mathcal{H})$ a C*-algebras with
unit and $\mathbf{V}$ an isometry on Hilbert space $\mathcal{H}$ such that 
\begin{equation*}
\mathbf{V}^*\mathfrak{A}\mathbf{V}\subset\mathfrak{A}.
\end{equation*}
If $(\widehat{\mathbf{V}},\widehat{\mathcal{H}},\mathbf{Z})$ denotes the minimal unitary dilation of the isometry $\mathbf{V}$ we shall prove the following proposition:

\begin{proposition} 
There exists a C*-algebra with unit 
$\widehat {\mathfrak{A}}\subset\mathfrak{B}(\widehat{\mathcal{H}}) $ such that:
\newline
1 - $\mathbf{Z}\mathfrak{A}\mathbf{Z}^*\subset\widehat{\mathfrak{A}}$ \ \ and \ \ $\mathbf{Z}^*\widehat{\mathfrak{A}}\mathbf{Z}\subset\mathfrak{A,}$
\newline
2 - $\widehat{\mathbf{V}}^*\widehat{\mathfrak{A}}\widehat{\mathbf{V}}
\subset\widehat{\mathfrak{A}}$,
\newline
3 - $\mathbf{Z}^*\widehat{\mathbf{V}}^*X\widehat{\mathbf{V}}\mathbf{Z=\mathbf{V}^*Z}^*X
\mathbf{ZV,}$ \ \ \ \  for all $X\in\widehat {\mathfrak{A}}$,
\newline
4 - $\mathbf{Z}^*\widehat{\mathbf{V}}^*(\mathbf{Z}A\mathbf{Z}^*)\widehat{\mathbf{V}}=\mathbf{V}^*A
\mathbf{V}$, \ \ \ \ \ for all $A\in\ \mathfrak{A}$.\label{prop-invariantalgebra}
\end{proposition}
first of all we want to consider some special operators on Hilbert space 
$\mathcal{H}$.

\subsection{The gamma operators associated to pair $\left(\mathfrak{A},V\right) $}
The sequences of elements of type 
$ \alpha =( n_{1},n_{2}....n_{r},A_{1},A_{2}...A_{r})$, 
with 
$n_{j}\in \mathbb{N}$ 
and
$\ A_{j}\in \mathfrak{A}$ 
for all 
$j=1,2...r$,  
are called strings of 
$\mathfrak{A}$ 
of length $r$ and weight 
$\sum\limits_{i=1}^{n}n_{i}$. 
\newline
For each $\alpha $ string of $\mathfrak{A}$, we associate the following
operators of 
$\mathfrak{B}(\mathcal{H})$:

\begin{equation*}
|\alpha) =A_{1}\mathbf{V}^{n_{1}}\cdot \cdot \cdot A_{r}\mathbf{V}\text{ \ \ \ and\ \ \ \ } (\alpha| =\mathbf{V}^{n_{r}^*}A_{r}\cdot \cdot \cdot \mathbf{V}^{n_{1}^*}A_{1},
\end{equation*}
furthermore 
$\overset{.}{\alpha}= \sum\limits_{i=1}^{n}n_{i}$  
and 
$l(\alpha)=r$, while $|n)$ denote the set operators 
$|\alpha)$ 
with 
$\overset{\cdot}{\alpha }=n$ 
and usually
 
\begin{equation*}
|n) \mathfrak{A}=\left\{|\alpha) A:A\in 
\mathfrak{A}\text{ and }\alpha \text{-string of }\mathfrak{A}\text{ with }
\overset{\cdot }{\alpha }=n\right\} .
\end{equation*}
The symbols $\left( n\right\vert $ and $\mathfrak{A}\left( n\right\vert $ \
have the same obvious meaning of above.

\begin{proposition}
Let $\alpha $ and $\beta $ are strings of $\mathfrak{A}$ for each 
$R\in\mathfrak{A}$ 
we have: 
\begin{equation}
\left( \alpha \right\vert R\left\vert \beta \right) \in \left\{ 
\begin{array}{cc}
\mathfrak{A}\left(\overset{\cdot}{\alpha }-\overset{\cdot}{\beta}
\right\vert&\text{if \ }\overset{\cdot }{\alpha }\geq \overset{\cdot}{
\beta} \\ 
\left\vert \overset{\cdot }{\beta}-\overset{\cdot }{\alpha }\right) 
\mathfrak{A} & \text{if \ }\overset{\cdot }{\alpha }<\overset{\cdot }{\beta }%
\end{array}
\right. ,  
\label{eq-napla-1}
\end{equation}
and with a simple calculation 
\begin{equation}
\left\vert \alpha \right) R\left\vert \beta \right) \in \left\vert \overset{\cdot }{\alpha }+\overset{\cdot }{\beta }\right) .
\label{eq-napla-2}
\end{equation}
\end{proposition}
\begin {proof}
For each $m,n\in \mathbb{N}$ and $R\in \mathfrak{A}$ we have:
\begin{equation}
\mathbf{V}^{m^{\ast }}R\mathbf{V}^{n}\mathbf{\in }\left\{ 
\begin{array}{cc}
\mathbf{V}^{\left( m-n\right)^{\ast }}\mathfrak{A} & m\geq n \\ 
\mathfrak{A}\mathbf{V}^{\left( n-m\right)} & m<n
\end{array}
\right .  
\label{eq-napla-3}
\end{equation}
Let 
$\alpha =( m_{1},m_{2}....m_{r},A_{1},A_{2}...A_{r}) $ 
and 
$\beta =( n_{1},n_{2}....n_{s},B_{1},B_{2}...B_{S}) $ 
 strings of $\mathfrak{A}$,
 we obtain:
\begin{equation*}
\left( \alpha \right\vert R\left\vert \beta \right) =\mathbf{V}^{m_{r}^{\ast
}}A_{r}\cdot \cdot \cdot \mathbf{V}^{m_{1}^{\ast }}A_{1}RB_{1}\mathbf{V}%
^{n_{1}}\cdot \cdot \cdot B_{s}\mathbf{V}^{n_{s}}=\left( \widetilde{\alpha }%
\right\vert \mathbf{I}\left\vert \widetilde{\beta }\right) 
\end{equation*}
where 
$\widetilde{\alpha }$ 
and 
$\widetilde{\beta }$ 
are strings of 
$\mathfrak{A}$ 
with  
$l\left(\widetilde{\alpha }\right)+l\left(\widetilde{\beta }\right)=l\left( \alpha \right) +l\left( \beta \right) -1$. 
Moreover if 
$\overset{\cdot }{\alpha }\geq \overset{\cdot }{\beta }$ 
we have 
$\overset{\cdot }{\widetilde{\alpha }}\geq \overset{\cdot }{\widetilde{\beta }}$
while if  
$\overset{\cdot }{\alpha }<\overset{\cdot }{\beta }$ 
it follows that 
$\overset{\cdot }{\widetilde{\alpha }}<\overset{\cdot }{\widetilde{
\beta }}$.
\newline
In fact if $m_{1}\geq n_{1}$ we obtain:
\begin{equation*}
\left( \alpha \right\vert R\left\vert \beta \right) =\mathbf{V}^{m_{r}^{\ast
}}A_{r}\cdot \cdot \cdot A_{2}\mathbf{V}^{\left( m_{1}-n_{1}\right) ^{\ast
}}R_{1}B_{2}\mathbf{V}^{n_{2}}\cdot \cdot \cdot B_{s}\mathbf{V}%
^{n_{s}}=\left( \widetilde{\alpha }\right\vert \mathbf{I}\left\vert 
\widetilde{\beta }\right),
\end{equation*}
where 
$R_{1}=\mathbf{V}^{n_{1}^{\ast }}A_{1}RB_{1}\mathbf{V}^{n_{1}},$ 
$\widetilde{\alpha }=\left(m_{1}-n_{1},m_{2}....m_{r},R_{1},A_{2}...A_{r}\right)$ 
and 
$\widetilde{\beta }=\left( n_{2}....n_{s},B_{2}...B_{S}\right)$.
\newline
If $m_{1}<n_{1}$ we can write:
\begin{equation*}
\left( \alpha \right\vert R\left\vert \beta \right) =\mathbf{V}^{m_{r}^{\ast
}}A_{r}\cdot \cdot \cdot \mathbf{V}^{m_{2}^{\ast }}A_{2}R_{1}\mathbf{V}%
^{\left( n_{1}-m_{1}\right) }B_{2}\cdot \cdot \cdot B_{s}\mathbf{V}%
^{n_{s}}=\left( \widetilde{\alpha }\right\vert \mathbf{I}\left\vert 
\widetilde{\beta }\right) ,
\end{equation*}%
where 
$R_{1}=\mathbf{V}^{m_{1}^{\ast }}A_{1}RB_{1}\mathbf{V}^{m_{1}},$ 
$\widetilde{\alpha }=\left( m_{2}....m_{r},A_{2}...A_{r}\right)$ 
and 
$\widetilde{\beta }=\left(n_{1}-m_{1},n_{2}....n_{s},R_{1},B_{2}...B_{S}\right)$.
\newline
Then by induction on number 
$\nu =l(\alpha) +l(\beta)$ 
we have the relationship \ref{eq-napla-1}.
\end {proof}
For each 
$\alpha $ 
string of 
$\mathfrak{A}$ 
with  
$\overset{\cdot }{\alpha}\geq1$, 
we define the linear operators:
\begin{equation*}
\Gamma(\alpha) =(\alpha\vert\mathbf{F\Pi}_{\overset{\cdot }{\alpha }-1},
\end{equation*}
that will be the gamma associated operators to the pair 
$(\mathfrak{A},\mathbf{V})$. 

\begin{proposition}
For each $\alpha $ and $\beta $ strings of $\mathfrak{A}$ with 
$\overset{\cdot}{\alpha},\overset{\cdot }{\beta }\geq 1$, 
the gamma operators associated to 
$ (\mathfrak{A},\mathbf{V}) $ 
satisfy the following relationship:
\begin{equation*}
\Gamma(\alpha)\cdot\Gamma(\beta)^*\in \mathfrak{A.}
\end{equation*}
\end{proposition}

\begin {proof}
We obtain:
\begin{equation*}
\Gamma(\alpha)\cdot\Gamma(\beta)^*=(\alpha\vert\mathbf{F\Pi }_{\overset{\cdot }{\alpha }-1}
\mathbf{\Pi }_{\overset{\cdot }{\beta }-1}^{^{\ast}}\mathbf{F}\left\vert
\beta \right) =\left\{ 
\begin{array}{cc}
\left( \alpha \right\vert \mathbf{F}\left\vert \beta \right)  & \overset{
\cdot }{\alpha }=\overset{\cdot }{\beta } \\ 
0 & \overset{\cdot }{\alpha }\neq \overset{\cdot }{\beta }
\end{array}
\right. ,
\end{equation*}
in fact
\begin{equation*}
 \left(\alpha \right\vert \mathbf{F}\left\vert \beta \right)=\left(\alpha
\right\vert \left( \mathbf{I}-\mathbf{VV}^{\ast }\right)\left\vert\alpha
\right)=\left( \alpha \right\vert \mathbf{I}\left\vert \alpha \right)
-\left( \alpha \right\vert \mathbf{VV}^{\ast }\left\vert \alpha \right) \in 
\mathfrak{A},
\end{equation*} 
since we have 
$\left(\alpha \right\vert\mathbf{V}\in \left(\overset{\cdot}{\alpha}-1\right\vert\ $ 
while 
$\mathbf{V}^{\ast}\left\vert\alpha
\right) \in \left\vert \overset{\cdot}{\alpha }-1\right) $ 
and by relationship \ref{eq-napla-1} follows that:
\begin{equation*}
\left(\overset{\cdot }{\alpha }-1\right\vert \mathbf{I}
\left\vert \overset{\cdot }{\alpha }-1\right) \subset \mathfrak{A}.
\end{equation*}
\end {proof}
We have an operator system 
$\Sigma $ 
of 
$\mathfrak{B}(l^{2}(\mathbf{F}\mathcal{H}))$ 
this is:
\begin{equation}
\Sigma =\left\{\mathbf{T}\in \mathfrak{B}( l^{2}(\mathbf{F}
\mathcal{H})) :\Gamma _{1}\mathbf{T}\Gamma _{2}^{\ast }\in 
\mathfrak{A}\ \  for\ all \ gamma \ operators \ \Gamma _{i} \ associated \ to \ (\mathfrak{A},\mathbf{V} \right\}. \label{eq-sigma-op}
\end{equation}

We observe that 
$\mathbf{I}\in \Sigma $ 
and 
$\Gamma_{1}^{\ast }\mathfrak{A}\Gamma _{2}\in \Sigma $ 
for all gamma operators 
$\Gamma_{i}.$ 
Moreover 
$ \Sigma $ 
is a norm closed, while it is a weakly closed if $\mathfrak{A}$ is
a W*-algebra.

\subsection{The napla operators}
For each 
$\alpha $, $\beta $ 
strings of 
$\mathfrak{A}$, $A\in \mathfrak{A}$
and 
$k\in \mathbb{N}$ 
we define the napla operators of 
$\mathfrak{B}(l^{2}( \mathbf{F}\mathcal{H}))$:
\begin{equation*}
\Delta_{k}(A,\alpha ,\beta) 
=\Pi_{\overset{\cdot}{\alpha}+k}^{^{\ast }}\mathbf{F}|\alpha)A(\beta| 
\mathbf{F}\Pi_{\overset{\cdot}{\beta }+k}.
\end{equation*}
For each $h,k\geq 0$ we obtain the following results: 
\begin{equation*}
\Delta_{k}(A,\alpha,\beta)^*=\Delta_{k}(A^*,\beta,\alpha) ,
\end{equation*}
and
\begin{equation}
\Delta_{k}(A,\alpha,\beta) \cdot \Delta _{h}(B,\gamma,\delta) =\left\{ 
\begin{array}{ccc}
0 & k+\overset{.}{\beta }\neq h+\overset{.}{\gamma }, &  \\ 
\Delta _{k}\left( R,\alpha ,\vartheta \right) & k+\overset{.}{\beta }=h+
\overset{.}{\gamma }, \ h-k\geq 0, \ with \  \overset{.}{\vartheta} =\overset{.}{\delta }+h-k \ and \ R\in \mathfrak{A} \\ 
\Delta _{h}\left( R,\vartheta ,\delta \right)  & k+\overset{.}{\beta }=h+%
\overset{.}{\gamma}, \ k-h>0, \ with \ \overset{.}{\vartheta} =\overset{.}{\delta }+k-h \ and \ R\in \mathfrak{A}
\end{array}
\right.  \label{eq-napla-4}
\end{equation}
In fact we have:
\begin{equation*}
\Delta _{k}\left( A,\alpha ,\beta \right) \cdot \Delta _{h}\left( B,\gamma
,\delta \right) =\Pi _{\overset{\cdot }{\alpha }+k}^{^{\ast }}\mathbf{F}
\left\vert \alpha \right) A\left( \beta \right\vert \mathbf{F}\Pi _{\overset{
\cdot}{\beta }+k}\Pi_{\overset{\cdot }{\gamma }+h}^{^{\ast }}\mathbf{F}
\left\vert \gamma \right) B\left( \delta \right\vert \mathbf{F}\Pi _{\overset
{\cdot}{\delta}+h}
\end{equation*}
and if 
$k+\overset{.}{\beta }\neq h+\overset{.}{\gamma }$ 
follows that 
$\Pi_{\overset{\cdot }{\beta }+k}\Pi_{\overset{\cdot }{\gamma }+h}^{^{\ast}}=0,$
while if 
$k+\overset{.}{\beta }=h+\overset{.}{\gamma }$, 
without losing generality we can get 
$h\geq k$, 
and we obtain
$\overset{.}{\beta }=\overset{.}{\gamma }+h-k\geq \overset{.}{\gamma }$ .
Moreover by relationship \ref{eq-napla-1}
\begin{equation*}
\left( \beta \right\vert \mathbf{F}\left\vert \gamma \right) \in \mathfrak{A}%
\left( \overset{.}{\beta }-\overset{.}{\gamma }\right\vert 
\end{equation*}
then 
\begin{equation*}
A\left(\beta \right\vert \mathbf{F}\left\vert \gamma \right) B\left(\delta
\right\vert \in \mathfrak{A}\left(\overset{\cdot}{\delta }+\overset{.}{
\beta }-\overset{.}{\gamma }\right\vert ,
\end{equation*}
there exists 
$\vartheta $ 
string of 
$\mathfrak{A}$ 
with 
$\overset{\cdot }{\vartheta }=\overset{\cdot }{\delta }+\overset{.}{\beta }-\overset{.}{\gamma}$ 
and a 
$R\in \mathfrak{A}$ 
such that:
\begin{equation*}
A\left( \beta \right\vert \mathbf{F}\left\vert \gamma \right) B\left( \delta
\right\vert =R\left( \vartheta \right\vert .
\end{equation*}
Since 
$\overset{\cdot}{\vartheta}=\overset{\cdot}{\delta}+h-k$ 
we have:
\begin{equation*}
\Delta _{k}\left(A,\alpha,\beta \right) \cdot \Delta _{h}\left( B,\gamma
,\delta \right) =\Pi _{\overset{\cdot }{\alpha }+k}^{^{\ast}}\mathbf{F}\left\vert \alpha \right) R\left(\vartheta \right\vert \mathbf{F}\Pi_{\overset{\cdot }{\delta}+h}=\Pi_{\overset{\cdot}{\alpha }+k}^{^{\ast}}\mathbf{F}\left\vert\alpha \right) R\left(\vartheta \right\vert \mathbf{F}\Pi_{\overset{\cdot}{\vartheta}+k}=\Delta_{k}\left( R,\alpha,\vartheta
\right) .
\end{equation*}

\begin{proposition}
The linear space 
$\mathfrak{X}_{o}$ 
generated by napla operators, is a *-subalgebra of 
$\mathfrak{B}\left(l^{2}\left(\mathbf{F}\mathcal{H}\right)
\right)\ $ 
included in the operator systems $\Sigma $ defined in \ref{eq-sigma-op}.
\end{proposition}

\begin {proof}
From relationship 
\ref{eq-napla-4} 
the linear space 
$\mathfrak{X}_{o}$\ 
is a *-algebra. Moreover for each gamma operators 
$\Gamma\left(\alpha \right) $ 
and 
$\Gamma\left(\beta\right) $ 
we obtain:
\begin{equation*}
\Gamma\left(\alpha\right)\Delta _{k}\left(A,\gamma,\delta \right)
\Gamma\left(\beta \right)^{\ast }=\left(\alpha \right\vert\mathbf{
F\Pi }_{\overset{\cdot }{\alpha }-1}\Pi_{\overset{\cdot}{\gamma}
+k}^{^{\ast }}\mathbf{F}\left\vert\gamma \right) A\left(\delta \right\vert 
\mathbf{F}\Pi_{\overset{\cdot}{\delta}+k}\mathbf{\Pi}_{\overset{\cdot}{
\beta}-1}\mathbf{F}\left\vert \beta \right) \in \mathfrak{A},
\end{equation*}
since by the relationships \ref{eq-napla-1} and \ref{eq-napla-2} we have
\begin{equation*}
\left(\alpha \right\vert\mathbf{F\Pi }_{\overset{\cdot }{\alpha }-1}\Pi_{\overset{\cdot}{\gamma }+k}^{^{\ast}}\mathbf{F}\left\vert\gamma\right)
A\left(\delta\right\vert\mathbf{F}\Pi_{\overset{\cdot}{\delta}+k}
\mathbf{\Pi }_{\overset{\cdot}{\beta}-1}\mathbf{F}\left\vert\beta\right)
\in \left\{
\begin{array}{cc}
\left( k+1\right\vert\mathfrak{A}\left\vert k+1\right)  & \overset{.}{\alpha }-1=\overset{.}{\gamma }+k, \ \overset{.}{\beta }-1=\overset{.}{\delta }+k \\ 
\mathbf{0} & \text{elsewhere}%
\end{array}
\right.
\end{equation*}
In fact if 
$\overset{.}{\alpha }=\overset{.}{\gamma }+k+1\ $
we can write:
\begin{equation*}
\left( \alpha \right\vert \mathbf{F\Pi }_{\overset{\cdot }{\alpha }-1}\Pi _{%
\overset{\cdot }{\gamma }+k}^{^{\ast }}\mathbf{F}\left\vert \gamma \right)
=\left( \alpha \right\vert \mathbf{F}\left\vert \gamma \right) =\left(
\alpha \right\vert \mathbf{I}\left\vert \gamma \right) -\left( \alpha
\right\vert \mathbf{VV}^{\ast }\left\vert \gamma \right) \in \mathfrak{A}%
\left( k+1\right\vert 
\end{equation*}
since 
\begin{equation*}
\left( \alpha \right\vert \mathbf{I}\left\vert \gamma \right) \in \mathfrak{A
}\left( k+1\right\vert \text{ and }\left( \alpha \right\vert \mathbf{VV}
^{\ast }\left\vert \gamma \right) \in \mathfrak{A}\left( k+1\right\vert 
\end{equation*}
while if 
$\overset{.}{\beta }=\overset{.}{\delta }+k+1$ 
we obtain
\begin{equation*}
\left( \delta \right\vert \mathbf{F}\Pi _{\overset{\cdot }{\delta }+k}%
\mathbf{\Pi }_{\overset{\cdot }{\beta }-1}\mathbf{F}\left\vert \beta \right)
\in \left( k+1\right\vert \mathfrak{A}.
\end{equation*}

\end {proof}

\begin{corollary}
The *-algebra 
$\mathfrak{X}_{o}$ 
and the operator systems $\Sigma$ are $\mathbf{W}$-invariant:
\begin{equation*}
\mathbf{W}^{\ast}\mathfrak{X}_{o}\mathbf{W}\subset\mathfrak{X}_{o}\text{ \
and \ }\mathbf{W}^{\ast}\Sigma\mathbf{W}\subset\Sigma.
\end{equation*}
\end{corollary}

\begin {proof}
Let be $\mathbf{T}$ belong to $\Sigma $, for each gamma operators 
$\Gamma\left( \alpha \right) $ 
and 
$\Gamma \left(\beta \right) $ 
we have:
\begin{eqnarray*}
\Gamma \left(\alpha \right) \left( \mathbf{W}^{\ast }\mathbf{TW}\right)
\Gamma \left(\beta \right)^{\ast} &=&\left( \alpha \right\vert \mathbf{
F\Pi}_{\overset{\cdot }{\alpha}-1}\mathbf{W}^{\ast }\mathbf{TW\Pi}_{
\overset{\cdot }{\beta }-1}\mathbf{F}\left\vert\beta \right) = \\
&=&\left( \alpha \right\vert \mathbf{F\mathbf{\Pi }_{\overset{\cdot }{\alpha 
}-2}T\mathbf{\Pi }_{\overset{\cdot }{\beta }-2}F}\left\vert \beta \right)
\in \mathfrak{A}\mathbf{V}^{\ast }\Gamma _{1}\left( \alpha _{o}\right) 
\mathbf{T}\Gamma _{2}\left( \beta _{o}\right) \mathbf{V}\mathfrak{A}\subset 
\mathbf{V}^{\ast}\mathfrak{A}\mathbf{V\subset }\mathfrak{A.}
\end{eqnarray*}
where 
$\alpha _{o}$ 
and 
$\beta _{o}$ 
are strings of 
$\mathfrak{A}$ 
with 
$\overset{.}{\alpha _{o}}=\overset{.}{\alpha }-1$ and $\overset{.}{\beta _{o}}%
=\overset{.}{\beta }-1$.
\newline
In fact let 
$\alpha =\left( m_{1},m_{2}....m_{r},A_{1},A_{2}...A_{r}\right) $
by definition of gamma operator, there is 
$i\leq r$ with $m_{i}\geq 1$ 
such that
\begin{equation*}
\left(\alpha \right\vert \mathbf{F\mathbf{\Pi }_{\overset{\cdot}{\alpha }-2}=}A_{1}\cdot \cdot \cdot \cdot A_{i}\mathbf{V}^{\ast }\mathbf{\left(
\alpha _{o}\right\vert \mathbf{F\mathbf{\Pi }_{\overset{\cdot }{\alpha }-2}=}%
}A_{1}\cdot \cdot \cdot \cdot A_{i}\mathbf{V}^{\ast }\Gamma \left( \alpha
_{o}\right) ,
\end{equation*}
where 
$\alpha _{o}=\left(0,..0,m_{i}-1,m_{i+1}..m_{r},A_{1},A_{2}...A_{r}\right) $ with 
$\overset{\cdot}{\alpha }_{o}=\overset{\cdot }{\alpha }-1.$
\end {proof}

Let $\mathfrak{X}$ be the closure in norm of the *-algebra 
$\mathfrak{X}_{o}$. 
Since $\Sigma $ is a norm closed set, we have 
$\mathfrak{X}\subset\Sigma $ 
while if 
$\mathfrak{A}$ 
is a von Neumann algebra of 
$\mathfrak{B}\left( \mathcal{H}\right) $ 
then  $\Sigma $ is weakly closed and 
$ \mathfrak{X}_{o}^{\prime\prime}\subset\Sigma .$

\begin{proposition}
The set 
\begin{equation}
\mathcal{S}=\left\{ \left\vert 
\begin{array}{cc}
A & \Gamma _{1} \\ 
\Gamma _{2}^{\ast } & \mathbf{T}%
\end{array}
\right\vert :A\in \mathfrak{A}, \ \mathbf{T}\in \mathfrak{X} \ and \ 
\Gamma _{i}\text{ are gamma op.of }\left(\mathfrak{A},\mathbf{V}\right)
\right\} ,  \label{eq-S-operator-systems}
\end{equation}
is an operator system of 
$\mathfrak{B}\left( \widehat{\mathcal{H}}\right) $ 
such that:
\begin{equation*}
\widehat{\mathbf{V}}^{\ast }\mathcal{S}\widehat{\mathbf{V}}\subset \mathcal{S}.
\end{equation*}
Furthermore
\begin{equation*}
\widehat{\mathbf{V}}^{\ast}C^{\ast }\left( \mathcal{S}\right) \widehat{%
\mathbf{V}}\subset C^{\ast }\left(\mathcal{S}\right) ,
\end{equation*}
where 
$C^{\ast}\left(\mathcal{S}\right)$ 
is the C*-algebra generated by the set 
$\mathcal{S}.$
\end{proposition}

\begin {proof}
We obtain:
\begin{equation*}
\widehat{\mathbf{V}}^{\ast }\mathcal{S}\widehat{\mathbf{V}}=\left\vert 
\begin{array}{cc}
\mathbf{V}^{\ast }A\mathbf{V} & \mathbf{V}^{\ast }A\mathbf{C}\left( 1\right)
+\mathbf{V}^{\ast }\Gamma _{1}\mathbf{W} \\ 
\mathbf{C}\left( 1\right) ^{\ast }A\mathbf{V}+\mathbf{W}^{\ast }\Gamma
_{2}^{\ast }\mathbf{V} & \mathbf{C}\left( 1\right) ^{\ast }A\mathbf{C}\left(
1\right) +\mathbf{W}^{\ast }\Gamma _{2}^{\ast }\mathbf{C}\left( 1\right)+
\mathbf{C}\left( 1\right) ^{\ast }\Gamma _{1}\mathbf{W}+\mathbf{W}^{\mathbf{\ast}}\mathbf{TW}
\end{array}
\right\vert ,
\end{equation*}
where the operators 
$\mathbf{V}^{\ast}\Gamma \left( \alpha \right) \mathbf{W}$ 
and 
$\mathbf{V}^{\ast}A\mathbf{C}\left( 1\right) $ 
are gamma operators associated to pair 
$\left(\mathfrak{A}\mathbf{,V}\right) $, 
while  
$ \mathbf{C}\left(1\right)^{\ast}A\mathbf{C}\left(1\right)$, 
$\mathbf{C}\left(1\right)^{\ast}\Gamma \left(\alpha \right)\mathbf{W,}$ 
and 
$\mathbf{W}^{\mathbf{\ast}}T\mathbf{W}$ 
are operators belonging to 
$\mathfrak{X}$.
\newline
In fact we have the following relationships:
\begin{equation*}
\mathbf{V}^{\ast }A\mathbf{C}\left( 1\right) =\mathbf{V}^{\ast }A\mathbf{%
F\Pi }_{0}=\Gamma \left( \vartheta \right) \text{\ \ with}~\vartheta =\left(
1,A\right) .
\end{equation*}
while if 
$\alpha =\left( m_{1},m_{2}..m_{r},A_{1},A_{2}...A_{r}\right) $ 
we obtain:
\begin{equation*}
\mathbf{V}^{\ast }\Gamma \left(\alpha \right) \mathbf{W=V}^{\ast }\left(
\alpha \right\vert \mathbf{F\Pi }_{\overset{\cdot }{\alpha }-1}\mathbf{W=}%
\Gamma \left( \vartheta \right) ,
\end{equation*}
with 
$\vartheta =\left(m_{1}+1,m_{2}..m_{r},A_{1},A_{2}...A_{r}\right)$
since 
$\mathbf{\Pi}_{\overset{\cdot }{\alpha}-1}\mathbf{W=\Pi}_{\overset{\cdot}{\alpha}}$.
\newline
Furthermore
\begin{equation*}
\mathbf{C}\left( 1\right) ^{\ast }A\mathbf{C}\left( 1\right) =\mathbf{\Pi }%
_{0}^{^{\ast }}\mathbf{F}A\mathbf{F\Pi }_{0}=\Delta _{0}\left( A,\alpha
,\beta \right) \text{ \ with }\alpha =\beta =\left( 0,\mathbf{I}\right) 
\end{equation*}
while
\begin{equation*}
\mathbf{C}\left( 1\right) ^{\ast }\Gamma \left( \alpha \right) \mathbf{W=\Pi 
}_{0}^{^{\ast }}\mathbf{F}\left( \alpha \right\vert \mathbf{F}\Pi _{\overset{%
\cdot }{\alpha }-1}\mathbf{W=\Pi }_{0}^{^{\ast }}\mathbf{F}\left\vert \gamma
\right) \left( \alpha \right\vert \mathbf{F}\Pi _{\overset{\cdot }{\alpha }%
+0}=\Delta _{0}\left( \mathbf{I},\gamma ,\alpha \right) \text{ with }\gamma
=\left( 0,\mathbf{I}\right) \mathbf{.}
\end{equation*}
\end {proof}
We observe that the *-algebra 
$\mathcal{A}^{\ast}\left(\mathcal{S}\right) $ 
generated by the operator system $\mathcal{S}$ is given by 
\begin{equation}
\mathcal{A}^{\ast }\left( \mathcal{S}\right) =\left\vert 
\begin{array}{cc}
\mathfrak{A} & \mathfrak{A}\Gamma \mathfrak{X} \\ 
\mathfrak{X}\Gamma ^{\ast }\mathfrak{A} & \mathfrak{X}
\end{array}
\right\vert .  \label{S-Algebra}
\end{equation}

Now we can easily prove proposition \ref{prop-invariantalgebra}.

\begin {proof}
We get 
$C^{\ast}\left(\mathcal{S}\right) $, 
the C*-algebra generated by 
$\mathcal{S}$ 
defined in \ref{eq-S-operator-systems}, 
by the definition 
$\mathbf{Z}\mathfrak{A}\mathbf{Z}^{\ast }\subset \mathcal{S}$ 
then
\begin{equation*}
\mathbf{Z}^{\ast }C^{\ast}\left(\mathcal{S}\right) \mathbf{Z}\subset 
\mathfrak{A}\text{.}
\end{equation*}
Moreover for $X\in C^*(\mathcal{S})$ we have:
\begin{equation*}
\mathbf{Z}^*\widehat{\mathbf{V}}^*X\widehat{\mathbf{V}}\mathbf{Z=
\mathbf{V}Z}^*X\mathbf{ZV}, 
\end{equation*}%
since $\widehat{\mathbf{V}}\mathbf{Z=ZV} $.
\newline
Let be $\mathfrak{F}$ the family of C*-subalgebras 
$\widehat{\mathfrak{B}}$
with unit of 
$C^*(\mathcal{S}) $ such that $\mathbf{Z}\mathfrak{A}\mathbf{Z}^*\subset \widehat{\mathfrak{B}}$ and $\widehat{\mathbf{V}}^*\widehat{\mathfrak{B}}\widehat{\mathbf{V}}\subset \widehat{\mathfrak{B}}$.
The family $\mathfrak{F}$ with inclusion is partially ordered set, then for
Zorn lemma's exists a minimal element that we shall denote with 
$\widehat{\mathfrak{A}}.$
\end {proof}

\section{Stinespring's theorem and dilations}
We examine a concrete C*-algebra 
$\mathfrak{A}$ 
of 
$\mathcal{B}(\mathcal{H})$ 
with unit and an ucp-map 
$\Phi:\mathfrak{A}\rightarrow\mathfrak{A}.$ 
By the Stinespring theorem for the ucp-map $\Phi$, we can deduce a triple $(\mathbf{V}_{\Phi},\sigma_{\Phi},\mathcal{L}_{\Phi})$ 
constituted by a Hilbert space 
$\mathcal{L}_{\Phi}$, 
a representation 
$\sigma_{\Phi}:\mathfrak{A}\rightarrow\mathcal{B}(\mathcal{L}_{\Phi})  $ 
and a linear contraction $\mathbf{V}_{\Phi}:\mathcal{H}\rightarrow\mathcal{L}_{\Phi}$ 
such thata for $\in\mathfrak{A}$,
\begin{equation}
\Phi(a)=\mathbf{V}_{\Phi}^*\sigma_{\Phi}(a)\mathbf{V}_{\Phi}.
\label{stine}
\end{equation}
We recall that on the algebraic tensor 
$\mathfrak{A}\otimes\mathcal{H}$ 
we can define a semi-inner product by
\[
\left\langle a_{1}\otimes\Psi_{1},a_{2}\otimes\Psi_{2}\right\rangle_{\Phi}=\left\langle \Psi_{1},\Phi\left( a_{1}^{\ast}a_{2}\right)\Psi_{2}\right\rangle_{\mathcal{H}},
\]
for all 
$a_{1},a_{2}\in\mathfrak{A}$ 
and 
$\Psi_{1},\Psi_{2}\in\mathcal{H}$
furthermore the Hilbert space 
$\mathcal{L}_{\Phi}$ 
is the completion of the quotient space 
$\mathfrak{A}\overline{\otimes}_{\Phi}\mathcal{H}$
 of 
$\mathfrak{A}\otimes\mathcal{H}$ 
by the linear subspace
\[
\left\{X\in\mathfrak{A}\otimes\mathcal{H}:\left\langle X,X\right\rangle
_{\Phi}=0\right\}
\]
with inner product induced by 
$\left\langle \cdot,\cdot\right\rangle _{\Phi}$. 
We shall denote the image at 
$a\otimes\Psi\in\mathfrak{A}\otimes\mathcal{H}$ 
in 
$\mathfrak{A}\overline{\otimes}_{\Phi}\mathcal{H}$ 
by 
$a\overline{\otimes}_{\Phi}\Psi,$ 
so that we have
\[
\left\langle a_{1}\overline{\otimes}_{\Phi}\Psi_{2},a_{2}\overline{\otimes
}_{\Phi}\Psi_{2}\right\rangle _{\mathcal{L}_{\Phi}}=\left\langle \Psi_{1}
,\Phi\left(a_{1}^{\ast}a_{2}\right)\Psi_{2}\right\rangle _{\mathcal{H}},
\]
for all $a_{1},a_{2}\in\mathfrak{A}$ and $\Psi_{1},\Psi_{2}
\in\mathcal{H}$.
\newline
Moreover
$\sigma_{\Phi}\left(a\right)\left(x\overline{\otimes}_{\Phi}\Psi\right)  =ax\otimes_{\Phi}\Psi,$ 
for each 
$x\overline{\otimes}_{\Phi}\Psi\in\mathcal{L}_{\Phi}$ 
and 
$\mathbf{V}_{\Phi}\Psi=\mathbf{1}\overline{\otimes}_{\Phi}\Psi$ 
for each 
$\Psi\in\mathcal{H}.
$\newline 
Since $\Phi$ is unital map, the linear operator 
$\mathbf{V}_{\Phi}$ 
is an isometry with adjoint 
$\mathbf{V}_{\Phi}^{\ast}$ 
defined by
\[
\mathbf{V}_{\Phi}^*a\overline{\otimes}_{\Phi}\Psi=\Phi(a)\Psi,
\] 
for all $a\in\mathfrak{A}$ and $\Psi\in\mathcal{H}$.
\newline
We recall that the multiplicative domain of the ucp-map
$\Phi:\mathfrak{A}\rightarrow\mathfrak{A}$ 
is the C*-subalgebra of $\mathfrak{A}$ such defined:
\[
\mathcal{D}_{\Phi}=\{a\in\mathfrak{A}:\Phi(a^*)
\Phi(a)=\Phi(a^*a)\ \text{and}\ \Phi(
a)\Phi(a^*)=\Phi(aa^*)\},
\]
we have the following implications (See Paulsen Ref.\cite{pau 86}):
\newline  
$ a\in\mathcal{D}_{\Phi}$  
if and only if 
$\Phi(a)\Phi(x)=\Phi(ax)$ 
and  
$\Phi(x)\Phi(a)=\Phi(xa)$ 
for all
$x\in\mathfrak{A}$.

\begin{proposition}
The ucp-map $\Phi$ is a multiplicative if and only if $\mathbf{V}_{\Phi}$ is
an unitary. Moreover if 
$x\in\mathcal{D}\left(\Phi\right)$ 
we have:
\[
\sigma_{\Phi}\left(x\right)\mathbf{V}_{\Phi}\mathbf{V}_{\Phi}^{\ast}=
\mathbf{V}_{\Phi}\mathbf{V}_{\Phi}^{\ast}\sigma_{\Phi}\left(x\right)  .
\]

\end{proposition}

\begin {proof}
For each $\Psi\in\mathcal{H}$ we obtain the following implications:
\[
a\overline{\otimes}_{\Phi}\Psi=\mathbf{1}\overline{\otimes}_{\Phi}\Phi\left(
a\right)\Psi\ \ \ \Longleftrightarrow\ \ \ \Phi\left(a^{\ast}a\right)
=\Phi\left( a^{\ast}\right)\Phi\left(a\right)  ,
\]
since
\[
\left\Vert a\overline{\otimes}_{\Phi}\Psi-1\overline{\otimes}_{\Phi}
\Phi\left(a\right)\Psi\right\Vert =\left\langle \Psi,\Phi\left(  a^{\ast
}a\right)\Psi\right\rangle -\left\langle \Psi,\Phi\left( a^{\ast}\right)
\Phi\left(a\right)\Psi\right\rangle .
\]
Furthermore, for each 
$a\in\mathfrak{A}$ 
and 
$\Psi\in\mathcal{H}$ 
we have
$\mathbf{V}_{\Phi}\mathbf{V}_{\Phi}^{\ast}a\overline{\otimes}_{\Phi}
\Psi=\mathbf{1}\overline{\otimes}_{\Phi}\Phi\left(a\right)\Psi.$
\end {proof}

Now we prove the following Stinespring-type theorem (See Zsido
Ref.\cite{Zsido}):

\begin{proposition}
Let $\mathfrak{A}$ be a concrete C*-subalgebra with unit of 
$\mathfrak{\mathcal{B}}\left(\mathcal{H}\right)$ 
and
$\Phi:\mathfrak{A}\rightarrow\mathfrak{A}$ 
an ucp-map, then there exists a faithful representation 
$\left(\pi_{\infty},\mathcal{H}_{\infty}\right)$
of 
$\mathfrak{A}$ 
and an isometry 
$\mathbf{V}_{\infty}$ 
on Hilbert Space
$\mathcal{H}_{\infty}$ 
such that for $ a\in\mathfrak{A}$,
\begin{equation}
\mathbf{V}_{\infty}^{\ast}\mathcal{\pi}_{\infty}\left(a\right)
\mathbf{V}_{\infty}=\mathcal{\pi}_{\infty}\left( \Phi\left(a\right)
\right), \label{stinespring dil}
\end{equation}
where
\[
\sigma_{0}=id,\text{ \ \ \ }\Phi_{n}=\sigma_{n}\circ\Phi\text{ }%
\]
and 
$\left(\mathbf{V}_{n},\sigma_{n+1},\mathcal{H}_{n+1}\right)$ 
is the Stinespring dilation of 
$\Phi_{n}$ 
for every 
$n\geq0,$
\begin{equation}
\mathcal{H}_{\infty}=\bigoplus_{j=0}^{\infty}\mathcal{H}_{j},\ \ \ \ \ \mathcal{H}_{j}=\mathfrak{A}\overline{\otimes}
_{\Phi_{j-1}}\mathcal{H}_{j-1},\ \ \text{\ for }j\geq1\text{ and }
\mathcal{H}_{0}=\mathcal{H}; \label{Hilbert space infinit}
\end{equation}
and
\[
\mathbf{V}_{\infty}(\Psi_{0},\Psi_{1},\Psi_{2},...)=(0,\mathbf{V}_{0}\Psi
_{0},\ \mathbf{V}_{1}\Psi_{1},...)
\]
for each 
$(\Psi_{0},\Psi_{1},\Psi_{2},...)\in\mathcal{H}_{\infty}.$
\newline
Furthermore the map $\Phi$ is a homomorphism if and only if 
$\mathbf{V}_{\infty}\mathbf{V}_{\infty}^{\ast}\in\mathcal{\pi}_{\infty}\left(
\mathfrak{A}\right)^{^{\prime}}$.
\label{prop-stinespring}
\end{proposition}

\begin {proof}
By the Stinespring theorem there is triple 
$(\mathbf{V}_{0},\sigma_{1},\mathcal{H}_{1})$ 
such that for each 
$a\in\mathfrak{A}$ 
we have
$\Phi(a)=\mathbf{V}_{0}^*\sigma_{1}(a)\mathbf{V}_{0}$. 
The application 
$a\in\mathfrak{A}\rightarrow\sigma_{1}(\Phi(a))\in\mathfrak{\mathcal{B}}(\mathcal{H}_{1})$ 
is a composition of cp-maps therefore it is also a cp map. Set
$\Phi_{1}(a)=\sigma_{1}(\Phi(a))$. 
By appling the Stinespring's theorem to 
$\Phi_{1}$, 
we have a new triple
$(\mathbf{V}_{1},\sigma_{2},\mathcal{H}_{2})$ 
such that
$\Phi_{1}(a)=\mathbf{V}_{1}^*\sigma_{2}(a)\mathbf{V}_{1}$. 
By induction for $n\geq1$ we define 
$\Phi_{n}(a)=\sigma_{n}(\Phi(a))$ 
and we have a triple 
$(\mathbf{V}_{n},\sigma_{n+1},\mathcal{H}_{n+1})$ 
such that
$\mathbf{V}_{n}:\mathcal{H}_{n}\rightarrow\mathcal{H}_{n+1}$ 
and 
$\Phi_{n}(a)=\mathbf{V}_{n}^*\sigma_{n+1}(a)
\mathbf{V}_{n}$. 
\newline We get the Hilbert space $\mathcal{H}_{\infty}$
defined in \ref{Hilbert space infinit} and the injective representation of
the C*-algebra 
$\mathfrak{A}$ 
on 
$\mathcal{H}_{\infty}$ :
\begin{equation}
\pi_{\infty}(a)=\bigoplus \limits_{n\geq0}
\sigma_{n}(a)\label{rap infinito}
\end{equation}
with $\mathcal{\sigma}_{0}(a)=a$,  for each $a\in\mathfrak{A}$.
\newline 
Let $\mathbf{V}_{\infty}:\mathcal{H}_{\infty}\rightarrow
\mathcal{H}_{\infty} $ 
be the isometry defined by
\begin{equation}
\mathbf{V}_{\infty}(\Psi_{0},\Psi_{1}....\Psi_{n}...)=(
0,\mathbf{V}_{0}\Psi_{0},\mathbf{V}_{1}\Psi_{1}....\mathbf{V}_{n}\Psi
_{n}...),
\label{isometria inf}%
\end{equation}
for all $\Psi_{i}\in\mathcal{H}_{i}$ with $i\in\mathbb{N}$.
\newline
The adjoint operator of 
$\mathbf{V}_{\infty}$ 
is
\begin{equation}
\mathbf{V}_{\infty}^{\ast}(\Psi_{0},\Psi_{1},....\Psi_{n}...)
=(\mathbf{V}_{0}^{\ast}\Psi_{1},\mathbf{V}_{1}^{\ast}\Psi
_{2}....\mathbf{V}_{n-1}^{\ast}\Psi_{n}...)  \label{isometria inf 2}
\end{equation}
for all $\Psi_{i}\in\mathcal{H}_{i}$ with $i\in\mathbb{N}$, therefore\begin{align*}
\mathbf{V}_{\infty}^{\ast}\pi_{\infty}\left(  a\right)  \mathbf{V}_{\infty}%
{\textstyle\bigoplus\limits_{_{n\geq0}}}
\Psi_{n}  &  =%
{\textstyle\bigoplus\limits_{_{n\geq0}}}
\mathbf{V}_{n}^{\ast}\mathcal{\sigma}_{n+1}\left(  a\right)  \mathbf{V}%
_{n}\Psi_{n}=%
{\textstyle\bigoplus\limits_{_{n\geq0}}}
\Phi_{n}\left(  a\right)  \Psi_{n}=\\
&  =%
{\textstyle\bigoplus\limits_{_{n\geq0}}}
\sigma_{n}\left( \Phi\left(a\right)  \right)  \Psi_{n}=\pi_{\infty}\left(
\Phi\left(  a\right)  \right)
{\textstyle\bigoplus\limits_{_{n\geq0}}}
\Psi_{n}.
\end{align*}
We notice that  
$\mathbf{E}_{n}=\mathbf{V}_{n}\mathbf{V}_{n}^{\ast}$ 
be the orthogonal projection of 
$\mathcal{B}\left(\mathcal{H}_{n-1}\right)$, 
we have:
\[
\mathbf{E}\left(  \Psi_{0},\Psi_{1}...\Psi_{n}..\right)  =\left(
0,\mathbf{E}_{0}\Psi_{1},\mathbf{E}_{1}\Psi_{2},...\mathbf{E}_{n}\Psi
_{n+1}...\right)  .
\]
Finally for the proof of the last statement we only need to note that $x$ belong to multiplicative domains 
$\mathcal{D}\left(\Phi\right)$ 
if and only if  we have:
\[
\pi_{\infty}\left(x\right)\mathbf{V}_{\mathbf{\infty}}\mathbf{V}_{\infty
}^{\ast}=\mathbf{V}_{\mathbf{\infty}}\mathbf{V}_{\infty}^{\ast}\pi_{\infty
}\left(x\right).
\].
\end {proof}
\begin{remark}
Let $(\mathfrak{M},\Phi)$ be a quantum process, the representation $\pi_{\infty}(a):\mathfrak{M}\rightarrow\mathfrak{B}(\mathcal{H_{\infty}})$ defined in proposition \ref{prop-stinespring} is normal, since the Stinespring representation $\sigma_{\Phi}:\mathfrak{A}\rightarrow\mathcal{B}(\mathcal{L}_{\Phi})$ is a normal map. Then $(\pi_{\infty},\mathcal{H}_{\infty},\mathbf{V}_{\infty})$ is a covariant representation of quantum process.
\end{remark}

\subsection{Dilations of ucp-Maps}
If 
$\left(\mathcal{H}_{\infty},\pi_{\infty},\mathbf{V}_{\infty}\right)$ 
is the Stinespring representation of proposition \ref{prop-stinespring}, we
have that 
$\mathbf{V}_{\infty}^{\ast}\pi_{\infty}\left(\mathfrak{A}\right)
\mathbf{V}_{\mathbf{\infty}}\subset\pi_{\infty}\left(\mathfrak{A}\right)$
and by proposition \ref{prop-invariantalgebra} there exists a C*-algebra with unit of 
$\mathfrak{\mathcal{B}}\left(\widehat{\mathcal{H}}\right)$ 
such that:
\newline1 - $\mathbf{Z}\pi_{\infty}\left(\mathfrak{A}\right)
\mathbf{Z}^{\ast}\subset\widehat{\mathfrak{A}},$
\newline2 - $\mathbf{Z}^{\ast
}\widehat{\mathfrak{A}}\mathbf{Z}=\pi_{\infty}\left(\mathfrak{A}\right)
\mathfrak{,}$
\newline3 - $\mathbf{Z}^*\widehat{\mathbf{V}}^*X\widehat{\mathbf{V}}\mathbf{Z=\mathbf{V}\pi_{\infty}\left(  \mathbf{Z}
^{\ast}X\mathbf{Z}\right)V}$, \ for all $\mathbf{X}\in\widehat{\mathfrak{A}}$.
\newline 
Furthermore, we have a homomorphism 
$\widehat{\Phi}:\widehat{\mathfrak{A}}\rightarrow\widehat{\mathfrak{A}}$ 
thus defined
\[
\widehat{\Phi}(X)=\widehat{\mathbf{V}}^*X\widehat{\mathbf{V}}
\]
for all $X\in\widehat{\mathfrak{A}}$, such that for $A\in\mathfrak{A}$, $X\in\widehat{\mathfrak{A}}$ and $n\in\mathbb{N}$ we have:
\[
\Phi^{n}(A)=\mathbf{Z}^*\widehat{\Phi}^{n}(\mathbf{Z}A\mathbf{Z}^*)\mathbf{Z}
\] 
and 
\[
\mathbf{Z}^*\widehat{\Phi}^{n}(X)\mathbf{Z}=\Phi^{n}(\mathbf{Z}^*X\mathbf{Z}).
\]
The quadruple
$(\widehat{\Phi},\widehat{\mathfrak{A}},\mathcal{H},\mathbf{Z})$
with the above properties, is said to be a multiplicative dilation
of ucp-map 
$\Phi:\mathfrak{A}\rightarrow\mathfrak{A}$.

\begin{remark}
It is clear that these results are easily extended to the von Neumann algebras $\mathfrak{M}$ with $\Phi$ normal ucp-map. In this way we obtain a dilation of discrete quantum process $(\mathfrak{M},\Phi)$.
\end{remark}

\section{Ergodic properties}
Let $\mathfrak{A}$ be a concrete C*-algebra of 
$\mathfrak{\mathcal{B}}\left(\mathcal{H}\right)$ 
with unit, 
$\Phi:\mathfrak{A}\rightarrow\mathfrak{A}$ 
an ucp-map and 
$\varphi$ 
a state on 
$\mathfrak{A}$ 
such that 
$\varphi\circ\Phi=\varphi$. We recall (See N.S.Z. Ref.\cite{NSZ}) that the state $\varphi$ is a ergodic
state, relative to the ucp-map $\Phi$, if
\[
\underset{n\rightarrow\infty}{\lim}\dfrac{1}{n+1}\sum \limits_{k=0}^{n}
[\varphi(a\Phi^{k}(b))-\varphi(a)\varphi(b)]=0,
\]
for all $a,b\in\mathfrak{A}$, while is weakly mixing if
\[
\underset{n\rightarrow\infty}{\lim}\dfrac{1}{n+1}\sum \limits_{k=0}^{n}
|\varphi(a\Phi^{k}(b))-\varphi(a)\varphi(b)| =0,
\]
for all $a,b\in\mathfrak{A}$.
\newline
We observe that by the Stinepring-type theorem \ref{prop-stinespring} we can assume, without losing generality, that $\mathfrak{A}$\ is a concrete C*-algebra of 
$\mathfrak{B}\left(\mathcal{H}\right)$, 
and that there is a linear isometry 
$\mathbf{V}$ 
on 
$\mathcal{H}$ 
such that:
\begin{equation*}
\Phi \left( A\right) =\mathbf{V}^{\ast }A\mathbf{V}\text{ for all }A\in 
\mathfrak{A.}
\end{equation*}%
Then 
$\left(\widehat{\mathbf{V}},\widehat{\mathcal{H}},\mathbf{Z}\right)$
is the minimal unitary dilation of 
$\left(\mathbf{V},\mathcal{H}\right)$ 
and the C*-algebra 
$\widehat{\mathfrak{A}}$ 
defined in proposition \ref{prop-invariantalgebra} is included in $\mathfrak{B}(\widehat{\mathcal{H}})$.
\newline
We want to prove the following ergodic theorem, for dilation ucp-map 
$(\widehat{\Phi},\widehat{\mathfrak{A}},\mathcal{H},\mathbf{Z})$
previously defined:

\begin{proposition}
\label{proposition apx 1}
If the ucp-map $\Phi$ admits a $\varphi$-adjoint and $\varphi$ is a ergodic state, we obtain:
\[
\underset{N\rightarrow\infty}{\lim}\dfrac{1}{N+1}\sum \limits_{k=0}^{N}
[\varphi(\mathbf{Z}^*X\widehat{\mathbf{V}}^{k^{\ast}}Y\widehat{\mathbf{V}}^{k}\mathbf{Z})  -\varphi(\mathbf{Z}^{\ast}X\mathbf{Z}\varphi(\mathbf{Z}^*Y\mathbf{Z})]  =0,
\]
while if $\varphi$ is weakly mixing:
\[
\underset{N\rightarrow\infty}{\lim}\dfrac{1}{N+1}\sum \limits_{k=0}^{N}
|\varphi(\mathbf{Z}^*X\widehat{\mathbf{V}}^{k^{\ast}}Y\widehat{\mathbf{V}}^{k}\mathbf{Z})  -\varphi(\mathbf{Z}^{\ast
}X\mathbf{Z})\varphi(\mathbf{Z}^*Y\mathbf{Z})|=0,
\]
for all $X,Y\in\widehat{\mathfrak{A}}.$
\end{proposition}
If we write every element $X$ of 
$\mathcal{B}\left(\widehat{\mathcal{H}}\right)$ 
in matrix form 
$X=\left\vert
\begin{array}
[c]{cc}%
X_{1,1} & X_{1,2}\\
X_{2,1} & X_{2,2}%
\end{array}
\right\vert $ 
with 
$\widehat{\mathcal{H}}=\mathcal{H}\oplus l^{2}\left(
\mathbf{F}\mathcal{H}\right)$ 
we obtain:
\[
\varphi\left(\mathbf{Z}^{\ast}X\widehat{\mathbf{V}}^{k^{\ast}}%
Y\widehat{\mathbf{V}}^{k}\mathbf{Z}\right)=\varphi\left(X_{1,1}%
\mathbf{V}^{k}Y_{1,1}\mathbf{V}^{k}\right)+\varphi\left(X_{1,2}%
\mathbf{C}\left(  k\right)^{\ast}Y_{1,1}\mathbf{V}^{k}\right)
+\varphi\left(X_{1,2}\mathbf{W}^{k^{\ast}}Y_{2,1}\mathbf{V}^{k}\right)
\]
and the proof of previous proposition is an easy consequence of the following lemma:

\begin{lemma}
\label{lemma apx 1}
Let $X\in\mathcal{A}^*(\mathcal{S})$, the *-algebra generated by operator system $\mathcal{S}$ defined in \ref{eq-S-operator-systems} and 
$Y\in\widehat{\mathfrak{A}}$,
\newline a] if $\varphi$ is an ergodic state
we have:
\begin{equation}
\underset{N\rightarrow\infty}{\lim}\dfrac{1}{N+1}%
{\textstyle\sum\limits_{k=0}^{N}}
\varphi\left(X_{1,2}\mathbf{C}\left(k\right)^{\ast}Y_{1,1}%
\mathbf{V}^{k}+X_{1,2}\mathbf{W}^{k^{\ast}}Y_{2,1}\mathbf{V}^{k}\right)=0,
\label{ergodica a]}%
\end{equation}
b] if $\varphi$ is weakly mixing we have:
\begin{equation}
\underset{N\rightarrow\infty}{\lim}\dfrac{1}{N+1}%
{\textstyle\sum\limits_{k=0}^{N}}
\left\vert \varphi\left(X_{1,2}\mathbf{C}\left(k\right)^{\ast
}Y_{1,1}\mathbf{V}^{k}+X_{1,2}\mathbf{W}^{k^{\ast}}Y_{2,1}\mathbf{V}%
^{k}\right)\right\vert =0. \label{ergodica b]}%
\end{equation}

\end{lemma}

\begin {proof}
Since 
$X\in\mathcal{A}^{\ast}\left(\mathcal{S}\right)$ 
we can assume that
$X_{1,2}=A\Gamma\left(\gamma\right)\Delta_{m}\left(B,\alpha ,\beta \right)$ 
with 
$A,B\in\mathfrak{A}$ 
and 
$\gamma$ 
string of 
$\mathfrak{A}$. 
Then:
\begin{equation}
X_{1,2}=A\left(\gamma\right\vert \mathbf{F\Pi}_{\overset{\cdot }{\gamma}%
-1}\Pi _{\overset{\cdot }{\alpha }+m}^{^{\ast }}\mathbf{F}\left\vert \alpha
\right) B\left( \beta \right\vert \mathbf{F}\Pi _{\overset{\cdot }{\beta}+m}=
\left\{\begin{array}{cc}
A\left(\gamma \right\vert \mathbf{F}\left\vert \alpha \right) B\left( \beta
\right\vert \mathbf{F}\Pi _{\overset{\cdot }{\beta}+m} & \overset{.}{\gamma }-1=\overset{.}{\alpha }+m \\ \mathbf{0} & \text{elsewhere}
\end{array}
\right. \label{eq-ergod1}
\end{equation}
Now we observe taht there is a natural number $k_{o}$ such that for each 
$k>k_{o} $ 
we obtain:
\[
X_{1,2}\mathbf{W}^{k^{\ast}}Y_{2,1}\mathbf{V}^{k}=0
\]
In fact we have that
\[
\mathbf{W}^{k^{\ast}}\left( \xi_{0},\xi_{1}...\xi_{n}...\right)  =\left(
\overset{k-time}{\overbrace{0,...0}},\xi_{0},\xi_{1}...\right)  ,
\]
for all 
$\left(\xi_{0},\xi_{1}...\xi_{n}..\right)  \in l^{2}\left(
\mathbf{F}\mathcal{H}\right)$ 
then 
$\Pi_{\beta+m}\mathbf{W}^{k^{\ast}}=\mathbf{0}$ 
for all 
$k>\overset{.}{\beta}+m$.
\newline 
It follows that:
\[
\underset{N\rightarrow\infty}{\lim}\dfrac{1}{N+1}%
{\textstyle\sum\limits_{k=0}^{N}}
\varphi\left(X_{1,2}\mathbf{C}\left(k\right)^{\ast}Y_{1,1}%
\mathbf{V}^{k}+X_{1,2}\mathbf{W}^{k^{\ast}}Y_{2,1}\mathbf{V}^{k}\right)=
\underset{N\rightarrow\infty}{\lim}\dfrac{1}{N+1}%
{\textstyle\sum\limits_{k=0}^{N}}
\varphi\left(X_{1,2}\mathbf{C}\left(k\right)^{\ast}Y_{1,1}
\mathbf{V}^{k}\right),
\]
Then we compute only the term 
$\varphi\left(X_{1,2}\mathbf{C}\left(k\right)^{\ast}Y_{1,1}
\mathbf{V}^{k}\right)$ 
and by relationship \ref{eq-ergod1} we can write that:
\[
X_{1,2}\mathbf{C}\left(k\right)^{\ast}Y_{1,1}\mathbf{V}^{k}=
A\left(\gamma \right\vert \mathbf{F}\left\vert \alpha \right) B\left(\beta
\right\vert \mathbf{F}\Pi _{\overset{\cdot }{\beta}+m}\mathbf{C}\left(k\right)^{\ast}Y_{1,1}\mathbf{V}^{k}
\]
moreover by relationship \ref{rel-2} for 
$k>\overset{.}{\beta}+m $ 
we have:
\[
\Pi _{\overset{\cdot }{\beta}+m}\mathbf{C}\left(k\right)^{\ast}=\mathbf{FV}
^{\left(k-\beta-m-1\right)^{\ast}},
\]
it follows that 
\[
X_{1,2}\mathbf{C}\left(k\right)^{\ast}Y_{1,1}\mathbf{V}^{k}
=
A\left(\gamma\right\vert\mathbf{F}\left\vert \alpha\right)B\left(\beta\right\vert \mathbf{FV}
^{\left(k-\beta-m-1\right)^{\ast}}Y_{1,1}\mathbf{V}^{k}
=
A\left(\gamma\right\vert \mathbf{F}\left\vert \alpha\right)B\left(
\beta\right\vert\mathbf{F}\Phi^{\left(k-\beta-1\right)}\left(
Y_{1,1}\right)\mathbf{V}^{\beta+m+1}.
\]
Since $\overset{\cdot }{\gamma }=\overset{\cdot }{\alpha }+m+1 $, by relationship \ref{eq-napla-1} we obtain:
\begin{equation*}
A\left(\gamma \right\vert \mathbf{F}\left\vert \alpha \right) B\left( \beta
\right\vert \in \mathfrak{A}\left( \overset{.}{\beta}+m+1\right\vert, 
\end{equation*}
it follows that there exists a 
$\vartheta $ 
string of 
$\mathfrak{A}$ 
with 
$\overset{\cdot}{\vartheta }=\overset{\cdot}{\beta}+m+1$ and an operator $R\in \mathfrak{A}$, 
such that
\begin{equation*}
A\left(\gamma \right\vert \mathbf{F}\left\vert \alpha \right)B\left(\beta
\right\vert 
=
R\left(\vartheta \right\vert .
\end{equation*} 
Then
\[
X_{1,2}\mathbf{C}\left(k\right)^{\ast}Y_{1,1}\mathbf{V}^{k}
=
R\left(\vartheta\right\vert\mathbf{F}\Phi^{\left(k-\beta-1\right)}\left(
Y_{1,1}\right)\mathbf{V}^{\beta+m+1}.
\]
If we set 
$\vartheta =\left(n_{1},n_{2},...n_{r},A_{1}\mathbf{,}A_{2},....A_{r}\right) .$
we have $n_{1}+n_{2}+...+n_{r}= \overset{.}{\beta}+m+1 $ and
\begin{equation*}
R\left(\vartheta\right\vert \mathbf{F}\Phi ^{\left(k-\overset{\cdot}{\beta}-1\right)}\left( Y_{1,1}\right) \mathbf{V}^{\overset{\cdot}{\beta}+m+1}
=
R\mathbf{\mathbf{V}}^{n_{r}^{\ast }}A_{r}\mathbf{\mathbf{V}}
^{n_{r-1}^{\ast }}A_{r-1}\cdot \cdot \cdot A_{2}\mathbf{\mathbf{V}}
^{n_{1}^{\ast }}A_{1}\mathbf{F}\Phi ^{\left( k-\overset{\cdot }{\beta}
-1\right)}\left( Y_{1,1}\right)\mathbf{V}^{\overset{\cdot}{\beta }+m+1}
=
\end{equation*}
\begin{equation*}
=R\Phi ^{n_{r}}\left( A_{r}\Phi ^{n_{r-1}}\left( A_{r-1}\cdot \cdot \cdot
\Phi ^{n_{2}}\left( A_{2}\mathbf{R}_{k}\right) \right) \right) ,
\end{equation*}%
where
\[
\mathbf{R}_{k}=\Phi ^{n_{r}}\left(A_{r}\right)\Phi^{\left( k-\beta
-1\right)}\left(Y_{1,1}\right)-\Phi ^{n_{r}-1}\left( \Phi\left(
A_{r}\right) \Phi ^{\left( k-\beta \right) }\left( Y_{1,1}\right) \right). 
\]
We have:
\newline
\newline
$\varphi \left( X_{1,2}\mathbf{C}\left( k\right)^{\ast
}Y_{1,1}\mathbf{V}^{k}\right) =\varphi \left( R\Phi ^{n_{r}}\left( A_{r}\Phi ^{n_{r-1}}\left(A_{r-1}\cdot \cdot \cdot \Phi ^{n_{2}}\left( A_{2}\mathbf{R}_{k}\right)
\right) \right) \right) =$

$=\varphi \left( \Phi _{\natural }^{n_{r}}\left( R\right) A_{r}\Phi
^{n_{r-1}}\left( A_{r-1}\left( \cdot \cdot \cdot \Phi ^{n_{2}}\left( A_{2}%
\mathbf{R}_{k}\right) \right. \right) \right) =$

$=\varphi \left( \Phi _{\natural }^{n_{r-1}}\left( \Phi _{\natural
}^{n_{r}}\left( R\right) A_{r}\right) A_{r-1}\left( A_{r-2}\cdot \cdot \cdot
A_{3}\Phi ^{n_{2}}\left( A_{2}\mathbf{R}_{k}\right) \right. \right) =$

$=\varphi \left( \Phi _{\natural }^{n_{2}}\left( \Phi _{\natural
}^{n_{3}}\cdot \cdot \cdot \Phi _{\natural }^{n_{r-1}}\left( \Phi _{\natural
}^{n_{r}}\left( R\right) A_{r}\right) \cdot \cdot \cdot A_{3}\right) A_{2}%
\mathbf{R}_{k}\right)$
\newline
\newline
and replacing $\mathbf{R}_{k}$, we obtain:
\newline
$\Phi _{\natural }^{n_{2}}\left( \Phi _{\natural }^{n_{3}}\cdot \cdot \cdot
\Phi _{\natural }^{n_{r-1}}\left( \Phi _{\natural }^{n_{r}}\left( R\right)
A_{r}\right) \cdot \cdot \cdot A_{3}\right) A_{2}\mathbf{R}_{k}=$

$=\Phi _{\natural }^{n_{2}}\left( \Phi _{\natural }^{n_{3}}\cdot \cdot \cdot
\Phi _{\natural }^{n_{r-1}}\left( \Phi _{\natural }^{n_{r}}\left( R\right)
A_{r}\right) \cdot \cdot \cdot A_{3}\right) A_{2}\Phi ^{n_{1}}\left(
A_{1}\right) \Phi ^{\left( k-\beta -1\right) }\left( Y_{1,1}\right) -$

$-\Phi _{\natural }^{n_{2}}\left( \Phi _{\natural }^{n_{3}}\cdot \cdot \cdot
\Phi _{\natural }^{n_{r-1}}\left( \Phi _{\natural }^{n_{r}}\left( R\right)
A_{r}\right) \cdot \cdot \cdot A_{3}\right) A_{2}\Phi ^{n_{1}-1}\left( \Phi
\left( A_{1}\right) \Phi ^{\left( k-\beta \right) }\left( Y_{1,1}\right)
\right) .$
\newline
\newline
Then:
\newline
$\varphi \left( X_{1,2}\mathbf{C}\left( k\right) ^{\ast }Y_{1,1}\mathbf{V%
}^{k}\right) =$

$=\varphi \left( \Phi _{\natural }^{n_{2}}\left( \Phi _{\natural
}^{n_{3}}\cdot \cdot \cdot \Phi _{\natural }^{n_{r-1}}\left( \Phi _{\natural
}^{n_{r}}\left( R\right) A_{r}\right) \cdot \cdot \cdot A_{3}\right)
A_{2}\Phi ^{n_{1}}\left( A_{1}\right) \Phi ^{\left( k-\beta -1\right)
}\left( Y_{1,1}\right) \right) -$

$-\varphi \left( \Phi _{\natural }^{n_{2}}\left( \Phi _{\natural
}^{n_{3}}\cdot \cdot \cdot \Phi _{\natural }^{n_{r-1}}\left( \Phi _{\natural
}^{n_{r}}\left( R\right) A_{r}\right) \cdot \cdot \cdot A_{3}\right)
A_{2}\Phi ^{n_{1}-1}\left( \Phi \left( A_{1}\right) \Phi ^{\left( k-\beta
\right) }\left( Y_{1,1}\right) \right) \right)$.
\newline
\newline
It follows that :
\newline
$\dfrac{1}{N+1}\sum\limits_{k=0}^{N}\varphi \left( X_{1,2}\mathbf{C}\left( k\right) ^{\ast }Y_{1,1}\mathbf{V}^{k}\right) = $
\newline
$=\dfrac{1}{N+1}\sum\limits_{k=0}^{N}\varphi \left( \Phi _{\natural
}^{n_{2}}\left( \Phi _{\natural }^{n_{3}}\cdot \cdot \cdot \Phi _{\natural
}^{n_{r-1}}\left( \Phi _{\natural }^{n_{r}}\left( R\right) A_{r}\right)
\cdot \cdot \cdot A_{3}\right) A_{2}\Phi ^{n_{1}}\left( A_{1}\right) \Phi
^{\left( k-\beta -1\right) }\left( Y_{1,1}\right) \right) -$\newline
$-\dfrac{1}{N+1}\sum\limits_{k=0}^{N}\varphi \left( \Phi _{\natural
}^{n_{2}}\left( \Phi _{\natural }^{n_{3}}\cdot \cdot \cdot \Phi _{\natural
}^{n_{r-1}}\left( \Phi _{\natural }^{n_{r}}\left( R\right) A_{r}\right)
\cdot \cdot \cdot A_{3}\right) A_{2}\Phi ^{n_{1}-1}\left( \Phi \left(
A_{1}\right) \Phi ^{\left( k-\beta \right) }\left( Y_{1,1}\right) \right)
\right) $.
\newline
\newline
If the state $\varphi$ is ergodic we have:
\newline
$\underset{N\rightarrow \infty }{\lim }\dfrac{1}{N+1}\sum\limits_{k=0}^{N}
\varphi \left( \Phi _{\natural }^{n_{2}}\left( \Phi _{\natural
}^{n_{3}}\cdot \cdot \cdot \Phi _{\natural }^{n_{r-1}}\left( \Phi_{\natural
}^{n_{r}}\left( R\right) A_{r}\right) \cdot \cdot \cdot A_{3}\right)
A_{2}\Phi ^{n_{1}-1}\left( \Phi \left( A_{1}\right) \Phi ^{\left( k-\beta
\right) }\left( Y_{1,1}\right) \right) \right) =$

$=\varphi \left( \Phi _{\natural }^{n_{2}}\left( \Phi _{\natural
}^{n_{3}}\cdot \cdot \cdot \Phi _{\natural }^{n_{r-1}}\left( \Phi _{\natural
}^{n_{r}}\left( R\right) A_{r}\right) \cdot \cdot \cdot A_{3}\right)
A_{2}\Phi ^{n_{1}}\left( A_{1}\right) \right) \varphi \left( Y_{1,1}\right) = $

$=\varphi \left( \Phi _{\natural }^{n_{1}}\left( \Phi _{\natural
}^{n_{2}}\left( \Phi _{\natural }^{n_{3}}\cdot \cdot \cdot \Phi _{\natural
}^{n_{r-1}}\left( \Phi _{\natural }^{n_{r}}\left( R\right) A_{r}\right)
\cdot \cdot \cdot A_{3}\right) A_{2}\right) A_{1}\right) \varphi \left(
Y_{1,1}\right) $
\newline
\newline
while
\newline
$\underset{N\rightarrow \infty }{\lim }\dfrac{1}{N+1}\sum\limits_{k=0}^{N}%
\varphi \left( \Phi _{\natural }^{n_{1}-1}\left( \Phi _{\natural
}^{n_{2}}\left( \Phi _{\natural }^{n_{3}}\cdot \cdot \cdot \Phi _{\natural
}^{n_{r-1}}\left( \Phi _{\natural }^{n_{r}}\left( R\right) A_{r}\right)
\cdot \cdot \cdot A_{3}\right) A_{2}\right) \Phi \left( A_{1}\right) \Phi
^{\left( k-\beta \right) }\left( Y_{1,1}\right) \right) =$\newline
$=\varphi \left( \Phi _{\natural }^{n_{1}-1}\left( \Phi _{\natural
}^{n_{2}}\left( \Phi _{\natural }^{n_{3}}\cdot \cdot \cdot \Phi _{\natural
}^{n_{r-1}}\left( \Phi _{\natural }^{n_{r}}\left( R\right) A_{r}\right)
\cdot \cdot \cdot A_{3}\right) A_{2}\right) \Phi \left( A_{1}\right) \right)
\varphi \left( Y_{1,1}\right) =$\newline
$=\varphi \left( \Phi _{\natural }\left( \Phi _{\natural }^{n_{1}-1}\left(
\Phi _{\natural }^{n_{2}}\left( \Phi _{\natural }^{n_{3}}\cdot \cdot \cdot
\Phi _{\natural }^{n_{r-1}}\left( \Phi _{\natural }^{n_{r}}\left( R\right)
A_{r}\right) \cdot \cdot \cdot A_{3}\right) A_{2}\right) \right)
A_{1}\right) \varphi \left( Y_{1,1}\right) ,$
\newline
\newline
then we obtain 
\newline
\newline
$ \underset{N\rightarrow \infty }{\lim }\dfrac{1}{N+1}\sum\limits_{k=0}^{N}
\varphi \left( X_{1,2}\mathbf{C}\left(k\right)^{\ast}Y_{1,1}\mathbf{V}^{k}\right) =0 $.
\newline
\newline
In weakly mixing case, using the previous results, we obtain:
\newline
\newline
$\left\vert \varphi \left( X_{1,2}\mathbf{C}_{k}^{\ast }Y_{1,1}\mathbf{V}
^{k}\right) \right\vert =
\left\vert \varphi \left(B\Phi ^{n_{1}}\left( A_{1}\right) \Phi ^{\left(
k-\overset{\cdot }{\beta }-1\right) }\left( Y_{1,1}\right) \right) -\varphi
\left( B\Phi ^{n_{1}-1}\left( \Phi \left( A_{1}\right)\Phi ^{\left( k-
\overset{\cdot }{\beta }\right) }\left( Y_{1,1}\right) \right)\right)
\right\vert $
\newline
\newline

where 
$B=\Phi _{\natural }^{n_{2}}\left( \Phi _{\natural }^{n_{3}}\cdot \cdot
\cdot \Phi _{\natural }^{n_{r-1}}\left( \Phi _{\natural }^{n_{r}}\left(
R\right) A_{r}\right) \cdot \cdot \cdot A_{3}\right) A_{2}$.
\newline
\newline
Adding and subtracting the element 
$\varphi \left( B\Phi ^{n_{1}}\left(
A_{1}\right) \right) \varphi \left( Y_{1,1}\right) $ we can write:
\newline
\newline
$\left\vert \varphi \left( B\Phi ^{n_{1}}\left( A_{1}\right) \Phi ^{\left( k-%
\overset{\cdot }{\beta }-1\right) }\left( Y_{1,1}\right) \right) -\varphi
\left( B\Phi ^{n_{1}-1}\left( \Phi \left( A_{1}\right) \Phi ^{\left( k-%
\overset{\cdot }{\beta }\right) }\left( Y_{1,1}\right) \right) \right)
\right\vert \leq $
\newline
$\leq \left\vert \varphi \left( B\Phi ^{n_{1}}\left( A_{1}\right) \Phi
^{\left( k-\overset{\cdot }{\beta }-1\right) }\left( Y_{1,1}\right) \right)
-\varphi \left( B\Phi ^{n_{1}}\left( A_{1}\right) \right) \varphi \left(
Y_{1,1}\right) \right\vert +$
\newline
$+\left\vert \varphi \left( B\Phi ^{n_{1}-1}\left( \Phi \left( A_{1}\right)
\Phi ^{\left( k-\overset{\cdot }{\beta }\right) }\left( Y_{1,1}\right)
\right) \right) -\varphi \left( B\Phi ^{n_{1}}\left( A_{1}\right) \right)
\varphi \left( Y_{1,1}\right) \right\vert .$
\newline
\newline
Moreover
\newline
\newline
$\left\vert \varphi \left( B\Phi ^{n_{1}-1}\left( \Phi \left( A_{1}\right)
\Phi ^{\left( k-\overset{\cdot }{\beta }\right) }\left( Y_{1,1}\right)
\right) \right) -\varphi \left( B\Phi ^{n_{1}}\left( A_{1}\right) \right)
\varphi \left( Y_{1,1}\right) \right\vert =$
\newline
$=\left\vert \varphi \left( \Phi _{\natural }^{n_{1}-1}\left( B\right) \Phi
\left( A_{1}\right) \Phi ^{\left( k-\overset{\cdot }{\beta }\right) }\left(
Y_{1,1}\right) \right) -\varphi \left( \Phi _{\natural }^{n_{1}-1}\left(
B\right) \Phi \left( A_{1}\right) \right) \varphi \left( Y_{1,1}\right)
\right\vert $,
\newline
\newline
and by the weakly mixing properties we obtain:
\begin{equation*}
\underset{N\rightarrow \infty }{\lim }\dfrac{1}{N+1}\sum\limits_{k=0}^{N}%
\left\vert \varphi \left( B\Phi ^{n_{1}}\left( A_{1}\right) \Phi ^{\left( k-%
\overset{\cdot }{\beta }-1\right) }\left( Y_{1,1}\right) \right) -\varphi
\left( B\Phi ^{n_{1}}\left( A_{1}\right) \right) \varphi \left(
Y_{1,1}\right) \right\vert =0,
\end{equation*}%
and%
\begin{equation*}
\underset{N\rightarrow \infty }{\lim }\dfrac{1}{N+1}\sum\limits_{k=0}^{N}%
\left\vert \varphi \left( \Phi _{\natural }^{n_{1}-1}\left( B\right) \Phi
\left( A_{1}\right) \Phi ^{\left( k-\overset{\cdot }{\beta }\right) }\left(
Y_{1,1}\right) \right) -\varphi \left( \Phi _{\natural }^{n_{1}-1}\left(
B\right) \Phi \left( A_{1}\right) \right) \varphi \left( Y_{1,1}\right)
\right\vert =0.
\end{equation*}
\end {proof}
Finally, the proof of proposition \ref {proposition apx 1} is a simple result of the previous lemma.


\begin{thebibliography}{9}                                                                                                

\bibitem {accardi} L. Accardi and C. Cecchini: \textit{Conditional expectations in von Neumann algebras and a theorem of Takesaki},J. Funct. An.\textbf{45} (1982) 245-273.

\bibitem {arv}W. Arveson: \textit{Non commutative dynamics and Eo-semigroups}. Monograph in mathematics. Springer-Verlag (2003).

\bibitem {bath}B.V. Bath and K.R. Parthasarathy: \textit{Markov dilations of nonconservative dynamical semigroups and quantum boundary theory}. Annales de l'I. H. P., section B, tome 31, No 4 (1995) 601-651

\bibitem {kum}B. K\"{u}mmerer: \textit{Markov dilations on W*-algebras}. - J.
Funct. Anal. \textbf{63} (1985), 139-177 .

\bibitem {Maj}W.A. Majewski: \textit{On the relationship between the reversibility of dynamics and balance conditions} - Annales de l'I. H. P. section A, tome \textbf{39}, no.1 (1983), 45-54.

\bibitem {Muh}P.S. Muhly and B. Solel: \textit{Quantum Markov Processes (correspondeces and dilations)}. - Int. J. Math Vol.13, No. 8 (2002), 863-906.

\bibitem {NF}B.Sz. Nagy and C. Foia\c{s}: \textit{Harmonic analysis of
operators on Hilbert space} - Regional Conference Series in Mathematics,
n.\textbf{19} (1971).

\bibitem {NSZ}c. Niculescu, A. Str\"{o}h and L.Zsid\'{o}: \textit{Non
commutative extensions of classical and multiple recurrence theorems} - J.
Operator Theory \textbf{50 }(2002), 3-52.

\bibitem {pau 86}V.I. Paulsen: \textit{Completely bounded maps and dilations} - Pitman Research Notes in Mathematics 146, Longman Scientific \& Technical, 1986.

\bibitem {Stine}F. Stinesring: \textit{Positive functions on C* algebras} - Proc. Amer. Math. Soc. \textbf{6} (1955) 211-216.

\bibitem {Zsido}L. Zsido: Personal communication - 2008.

\end{thebibliography}
\end{document}